\documentclass{amsart}

\usepackage[english]{babel} 
\usepackage[utf8]{inputenc} 
\usepackage[T1]{fontenc}
\usepackage{lmodern} 
\usepackage{thmtools} 
\usepackage{mathtools}
\usepackage{mleftright}
\usepackage{amssymb} 
\usepackage{hyperref}
\usepackage{tikz-cd}
\usepackage{floatrow}
\usepackage{graphicx}
\usepackage{caption}
\usepackage{enumitem}
\usepackage{todonotes}
\usepackage{mleftright}
\usepackage{cleveref} 

\declaretheorem[
name         = Theorem,
refname      = {Theorem,Theorems},
Refname      = {Theorem,Theorems},
numberwithin = section,
]{theorem}

\declaretheorem[
name         = Proposition,
refname      = {Proposition,Propositions},
Refname      = {Proposition,Propositions},
sibling      = theorem,
]{proposition}

\declaretheorem[
name         = Lemma,
refname      = {Lemma,Lemmas},
Refname      = {Lemma,Lemmas},
sibling      = theorem,
]{lemma}

\declaretheorem[
name         = Corollary,
refname      = {Corollary,Corollaries},
Refname      = {Corollary,Corollaries},
sibling      = theorem,
]{corollary}

\declaretheorem[
name        = Question,
refname     = {Question,Questions},
Refname     = {Question,Questions},
sibling     = theorem,
]{question}

\theoremstyle{definition}

\declaretheorem[
name        = Definition,
refname     = {Definition,Definitions},
Refname     = {Definition,Definitions},
sibling     = theorem,
]{definition}

\declaretheorem[
name        = Notation,
refname     = {Notation,Notations},
Refname     = {Notation,Notations},
sibling     = theorem,
]{notation}

\declaretheorem[
name        = Example,
refname     = {Example,Examples},
Refname     = {Example,Examples},
sibling     = theorem,
]{example}

\declaretheorem[
name        = Remark,
refname     = {Remark,Remarks},
Refname     = {Remark,Remarks},
sibling     = theorem,
]{remark}

\newlist{tfae}{enumerate}{1}%
\setlist[tfae,1]{label=(\roman*)}%

\newcommand{\cat}[1]{\mathsf{#1}}

\newcommand{\Quot}{\mathbf{Q}} 
\newcommand{\QuotEquiv}{\tilde{\mathbf{Q}}}
\newcommand{\Subm}{\mathbf{S}} 
\newcommand{\op}{\cat{op}}
\newcommand{\SMKH}{\cat{MetCH_{sep}}}
\newcommand{\SMKHop}{\cat{MetCH}_\cat{sep}^\op}
\newcommand{\SSMKH}{\cat{MetCH_{sep,sym}}}
\newcommand{\SSMKHop}{\cat{MetCH_{sep,sym}^\op}}
\newcommand{\CompHaus}{\cat{CH}}

\newcommand{\PosComp}{\cat{PosCH}}

\newcommand{\MKH}{\cat{MetCH}}

\newcommand{\epi}{\twoheadrightarrow}

\newcommand{\rmono}{\hookrightarrow}

\newcommand{\Set}{\cat{Set}}

\newcommand{\eps}{\varepsilon}

\newcommand{\upset}{\ensuremath{\mathord{\uparrow}\mkern1mu}} 
\newcommand{\downset}{\ensuremath{\mathord{\downarrow}\mkern1mu}} 

\newcommand{\black}{\color{black}}

\title{Barr-coexactness for metric compact Hausdorff spaces}

\author{Marco Abbadini}

\thanks{The first author's research was supported by the UK Research and Innovation (UKRI) under the UK government’s Horizon Europe funding guarantee (Project ``DCPOS'', grant number EP/Y015029/1) and by the Italian Ministry of University and Research through the PRIN project n.~20173WKCM5 \emph{Theory and applications of resource sensitive logics}}

\address{School of Computer Science, University of Birmingham, B15 2TT Birmingham, UK.}

\author{Dirk Hofmann}

\address{Center for Research and Development in Mathematics and Applications, Department of Mathematics, University of Aveiro, Portugal.}

\thanks{The second author acknowledges support by CIDMA under the FCT (Portuguese Foundation for Science and Technology) Multi-Annual Financing Program for R\&D Units.}

\keywords{Metric compact Hausdorff space, Stone-type duality, Nachbin space, compact ordered space, regular category, exact category, quantale enriched category}

\subjclass{%
  06F30, 
  54E45, 
  54F05, 
  18E08
}

\begin{document}

\begin{abstract}
  A \emph{metric compact Hausdorff space} is a Lawvere metric space equipped with a \emph{compatible} compact Hausdorff topology (which does not need to be the induced topology).
  These spaces maintain many important features of compact metric spaces, but the resulting category is much better behaved.

  In the category of separated metric compact Hausdorff spaces, we characterise the regular monomorphisms as the embeddings and the epimorphisms as the surjective morphisms.
  Moreover, we show that epimorphisms out of an object $X$ can be encoded internally on $X$ by their kernel metrics, which are characterised as the continuous metrics below the metric on $X$.
  Finally, as the main result, we prove that its dual category has an algebraic flavour: it is Barr-exact.
  While we show that it cannot be a variety of finitary algebras, it remains open whether it is an infinitary variety.
\end{abstract}

\maketitle

\tableofcontents

\section{Introduction}
\label{sec:introduction}

Since the early work of Pontrjagin \cite{Pon34} and Stone \cite{Sto36}, it is known that the duals of many categories of topology have an algebraic flavour: the category of compact Hausdorff Abelian groups is dually equivalent to the category of Abelian groups, and the category of Boolean spaces (particular compact Hausdorff spaces) is dually equivalent to the category of Boolean algebras.
Extending the latter fact, Duskin \cite{Dus69} observed that, besides the well-known fact that the category  \(\CompHaus\) of compact Hausdorff spaces and continuous maps is monadic over \(\Set\) \cite{Man69}, also its \emph{dual} category \(\CompHaus^\op\) is monadic over \(\Set\).
From a more concrete perspective, it is essentially shown in \cite{Gel41a} that \(\CompHaus^{\op}\) is equivalent to the category of commutative \(C^{*}\)-algebras and homomorphisms, and this category is indeed monadic over \(\Set\) with respect to the unit ball functor, as shown in \cite{Neg71}.
We also recall that, by the work of Stone \cite{Sto38} and Priestley \cite{Pri70, Pri72}, the category of Priestley spaces (Boolean spaces with a compatible partial order) and continuous monotone maps is dually equivalent to the category of distributive lattices and homomorphisms and therefore its dual category is also a variety.
Somewhat surprisingly, a similar investigation for related structures such as Nachbin's compact ordered spaces \cite{Nac65} was carried out only recently.
In \cite{HNN18} it is shown that the dual of the category \(\PosComp\) of compact ordered spaces and continuous monotone maps is a quasivariety, and in \cite{Abb19, AR20} it is finally shown that \(\PosComp^{\op}\) is also exact, and hence a variety.
In this paper we extend this line of research to include also metric structures.

Without doubt, the class of compact metric spaces (those metric spaces whose induced topology is compact) is an important class of metric spaces; however, together with non-expansive maps, this class forms a poorly behaved category.
Firstly, we cannot even form the coproduct of two singleton spaces, a shortcoming we can easily overcome by allowing the distance \(\infty\).
This modification allows us also to consider the sup-metric on an infinite product, which is indeed the product metric.
However, in general, the product metric does not induce the product topology and is therefore not necessarily compact.
For instance, for the two-element space \(2=\{0,1\}\) with distance \(1\) between \(0\) and \(1\), in \(2^{\mathbb{N}}\) the distance between two different points is also \(1\) and therefore, while the topological power \(2^{\mathbb{N}}\) is compact, the topology induced by the metric is discrete and hence non-compact.
We take this discrepancy as a motivation to consider not only metric spaces with an \emph{induced} compact (Hausdorff) topology but rather equipped with a \emph{compatible} compact Hausdorff topology.

This notion is also inspired by Nachbin's definition of a compact ordered space (see \cite{Nac65} and also \cite{Tho09}), where a compact Hausdorff space is equipped with a \emph{compatible} order relation.
In fact, ordered sets can be seen as a special case of metric structures once one drops the symmetry axiom from the definition of a metric space.
Accordingly, to include also the ordered case in our investigation, we consider here \emph{metric spaces} in a more general sense: a metric \(d\) on a set $X$ is a map $d \colon X\times X\to[0,\infty]$ which is only required to satisfy
\begin{displaymath}
  d(x,x)=0
  \quad\text{and}\quad
  d(x,z)\leq d(x,y)+d(y,z)
\end{displaymath}
for all \(x,y,z\in X\).
These two axioms are analogous to the reflexivity and the transitivity conditions of a preorder.
Under this analogy, the metric counterpart of anti-symmetry requires that, for all \(x,y\in X\), $d(x,y)=0=d(y,x)$ implies $x=y$; a metric space \((X,d)\) where \(d\) has this property is called \emph{separated}.
Let us note that the analogy between metric and order structures can be made more precise in the setting of enriched categories.
For instance, the transitivity condition of a preorder and the triangular inequality of a metric space can be both seen as an instance of the composition law of a category; similarly, reflexivity and the condition \(d(x,x)=0\) correspond to the identity law of a category; see \cite{Law73} for details.

In this paper we are interested in the category of compact Hausdorff spaces equipped with a \emph{compatible} separated metric (called separated metric compact Hausdorff spaces here, see \Cref{sec:prelimiaries} for details) and continuous non-expansive maps. This category constitutes a natural common roof for the category of compact ordered spaces and continuous monotone maps as well as the category of compact metric spaces and non-expansive maps.
Moreover, it has more pleasant properties than the latter one as it is, for instance, complete and cocomplete (see \cite{Tho09}).
Various properties and constructions of compact metric spaces and of compact ordered spaces can be naturally extended to this category. For instance, every metric compatible with some compact Hausdorff topology is Cauchy complete (see \cite{HR18}), and \cite{HN20} introduces a Hausdorff functor on this category combing naturally the Hausdorff metric and the Vietoris topology.
We also point out that this type of spaces proved to be useful in an extension of Stone-type dualities and of the notion of continuous lattice to metric structures (see \cite{GH13,HN18,HN23}).

Motivated by the corresponding results for compact Hausdorff spaces and compact ordered spaces, in this note we investigate the algebraic character of the dual of the category $\SMKH$ of separated metric compact Hausdorff spaces and continuous non-expansive maps.
We recall (see \cite{Bor94a}, for instance) that a category \(\mathsf{C}\) is \emph{regular} whenever \(\mathsf{C}\) is finitely complete, has pushouts of kernel pairs and regular epimorphisms are pullback-stable.
Moreover, \(\mathsf{C}\) is \emph{Barr-exact} whenever it is regular and every internal equivalence relation is effective, i.e., a kernel pair.
Being Barr-exact expresses an algebraic trait: for example, every variety of (possibly infinitary) algebras is Barr-exact; on the other hand, the category of topological spaces and continuous maps is not.
Barr-exactness distinguishes varieties from quasivarieties and can be seen as a categorical way to express the property of being ``closed under quotients of congruences''.
The main result of this paper states that the category \(\SMKH\) is Barr-coexact, that is, \(\SMKHop\) is Barr-exact.
Since \(\SMKH\) is known to be complete and cocomplete, to obtain this result we show that
\begin{itemize}
\item the regular monomorphisms and the epimorphisms in \(\SMKH\) are precisely the embeddings and the surjective morphisms, respectively (\Cref{s:quotient-objects}),
\item embeddings are stable under pushouts (\Cref{sec:coreg-separ-metr}), and
\item equivalence corelations are effective (\Cref{sec:barr-coex-separ}).
\end{itemize}

\section{Separated metric compact Hausdorff spaces}
\label{sec:prelimiaries}

We recall from the introduction that, in this paper, by a \emph{metric} we mean a map \(d \colon X\times X\to [0,\infty]\) such that, for all \(x,y,z\in X\),
  \begin{displaymath}
    d(x,x)=0
    \quad\text{and}\quad
    d(x,z)\leq d(x,y)+d(y,z).
  \end{displaymath}
  We say that \(d\) is \emph{separated} if, for all \(x,y\in X\), \(d(x,y)=0=d(y,x)\) implies \(x=y\).

We introduce the main objects of interest: separated metric compact Hausdorff spaces.
\begin{definition}[\cite{HR18}]
  \label{d:def:1}
  A \emph{(separated) metric compact Hausdorff space} \(X\) is a compact Hausdorff space \(X\) together with a (separated) metric  $d \colon X \times X \to [0,\infty]$ that is continuous with respect to the upper topology of $[0,\infty]$.
\end{definition}

We recall that the open subsets of the upper topology on \([0,\infty]\) are generated by the sets \(]u,\infty]\); hence, the non-empty closed subsets of \([0,\infty]\) are of the form $[0,u]$, with $u\in[0,\infty]$.
Throughout this paper we will often make use of the fact that, with respect to the upper topology, every non-empty compact subset of $[0,\infty]$ has a smallest element.
The space \([0,\infty]\) with the upper topology should be seen as the metric counterpart of the Sierpi{\'n}ski space \(\{0,1\}\) with \(\{1\}\) closed.
Accordingly, the notion of a separated metric compact Hausdorff space is the metric counterpart of Nachbin's compact ordered spaces.
A more general approach to ordered and metric (and other) topological structures is given in \cite{Tho09}, see \cref{r:algebras-for-ultrafilter-monad}.

The continuity with respect to the upper topology of $[0, \infty]$ of a function $f \colon X \to [0, \infty]$, with $X$ a topological space, is also known as \emph{lower semicontinuity} of $f$ (for instance, see \cite[IV.6.2]{Bourbaki}), and is equivalently described by the pointwise formula (for $x_0 \in X$)
\[
f(x_0) \leq \liminf_{x \to x_0} f(x)
\]
(or, equivalently, $f(x_0) = \liminf_{x \to x_0} f(x)$).
Therefore, the compatibility between the metric and the topology in \cref{d:def:1} amounts to the requirement that for all $x_0,y_0 \in X$ we have
\[
	d(x_0, y_0) \leq \liminf_{\substack{x \to x_0\\ y\to y_0}} d(x,y).
\]

  Every classical metric space whose induced topology is compact can be viewed as a separated metric compact Hausdorff space.
  More generally, every separated metric space whose induced symmetric topology is compact can be viewed as a separated metric compact Hausdorff space (see \cite{HR18}).
  Here we recall from \cite{Fla97} that the \emph{topology symmetrically induced by a metric \(d \colon X\times X\to [0,\infty]\)} is generated by the right and left open balls (with \(x_0\in X\) and \(u\in [0,\infty]\))
  \begin{displaymath}
    B_r(x_0,u)\coloneqq\{x\in X\mid d(x_0,x)<u\}
    \quad\text{and}\quad
    B_l(x_0,u)\coloneqq\{x\in X\mid d(x,x_0)<u\}.
  \end{displaymath}
  This topology is also generated by the symmetric open balls
  \begin{displaymath}
    B(x_0,u)\coloneqq B_r(x_0,u)\cap B_l(x_0,u)=
    \{x\in X\mid d(x_0,x)<u\;\text{and}\; d(x,x_0)<u\},
  \end{displaymath}
  with \(x_0\in X\) and \(u\in [0,\infty]\) (see \cite[Theorem~4.8]{Fla97}).

Next, we provide an example of a metric space whose induced topology is not compact, but which admits a natural compatible compact Hausdorff topology.

\begin{example}
  The interval \([0,\infty]\) becomes a separated metric space with the metric \(d\) defined by $d(u,v) = \max(v - u, 0)$ (see \cite{Law73}).
  This metric induces the topology on \([0,\infty]\) generated by the symmetric open balls (with \(u\in [0,\infty]\) and \(\varepsilon>0\))
  \begin{displaymath}
    \{v\in [0,\infty]\mid d(u,v)<\varepsilon \text{ and }d(v,u)<\varepsilon\}.
  \end{displaymath}
  We emphasise that this topology on \([0,\infty]\) is not compact, basically because $\infty$ is an isolated point.
  However, we may consider the \emph{Lawson topology} \cite{GHK+03} on \([0,\infty]\) which is generated by the basic open subsets $[0,u[$ and $]u,\infty]$, with $u\in [0,\infty]$.
  With respect to this topology, the interval \([0,\infty]\) is a compact Hausdorff space, and together with the metric \(d\) it becomes a separated metric compact Hausdorff space.
\end{example}

\begin{remark} \label{r:algebras-for-ultrafilter-monad}
  The definition of a metric compact Hausdorff space was pulled out of the hat in \cref{d:def:1}. However, it is not an arbitrary condition, but rather the specialisation to the metric setting of a more general notion. To explain the rationale behind it, the starting point is the observation made in \cite{Tho09} that the ultrafilter monad on \(\Set\) can be naturally extended to a monad on the category of preordered sets and monotone maps and on the category of (Lawvere) metric spaces and non-expansive maps, respectively.
  In the preordered case, the algebras for this monad are precisely Nachbin's preordered compact Hausdorff spaces \cite{Nac65}, and therefore the algebras in the metric case constitute a natural metric counterpart to preordered compact Hausdorff spaces.
  It is shown in \cite{HR18} that this algebraic description is equivalent to the condition in \cref{d:def:1}.
\end{remark}

\begin{definition}
	A map $f \colon X\to X'$ between metric spaces $(X,d)$ and $(X',d')$ is called \emph{non-expansive} if $d'(f(x),f(y))\leq d(x,y)$ for all $x,y\in X$.
	We denote the category of separated metric compact Hausdorff spaces and non-expansive continuous maps by \(\SMKH\).
\end{definition}

Below we collect some important properties of \(\SMKH\); for more details we refer to \cite{Tho09,GH13,HR18,HN20}.
\begin{theorem}\label{d:thm:1}
  The canonical forgetful functor \(\MKH\to\CompHaus\) from the category of metric compact Hausdorff spaces and continuous non-expansive maps to \(\CompHaus\) is topological.
  The full subcategory \(\SMKH\) of \(\MKH\) is reflective and therefore is complete and cocomplete.
  In particular:
  \begin{enumerate}
  \item The limit of a diagram \(D \colon\cat{I}\to\SMKH\) is given by the limit
    \begin{math}
      (p_i \colon X\to D(i))_{i\in\cat{I}}
    \end{math}
    in \(\CompHaus\) equipped with the \emph{sup-metric}
    \begin{displaymath}
      d(x,y)=\sup_{i\in\cat{I}}d_i(p_i(x),p_i(y)),
      \quad (x,y\in X),
    \end{displaymath}
    where \(d_i\) denotes the metric of the space \(D(i)\).
  \item The coproduct \(X=X_1+X_2\) of metric compact Hausdorff spaces \(X_1\) and \(X_2\) with metrics \(d_1\) and \(d_2\) is their disjoint union equipped with the coproduct topology and the coproduct metric, i.e., the metric \(d\) defined by
    \begin{displaymath}
      d(x,y)=
      \begin{cases}
        d_1(x,y) & \text{if \(x,y\in X_1\)},\\
        d_2(x,y) & \text{if \(x,y\in X_2\)},\\
        \infty & \text{otherwise}.
      \end{cases}
    \end{displaymath}
   \end{enumerate}
\end{theorem}

An \emph{embedding} in \(\SMKH\) is an injective morphism such that the metric on the domain is the restriction of the metric on the codomain.

\black
\begin{theorem}
  \label{d:thm:2}
  In the category \(\SMKH\), the pair
  \[(\{\text{surjective morphisms}\}, \{\text{embeddings}\})\]
   is an orthogonal factorisation system.
\end{theorem}

In~\Cref{sec:coreg-separ-metr} we will show that, in \(\SMKH\), the surjective morphisms are precisely the epimorphisms and the embeddings are precisely the regular monomorphisms (see \cref{p:rmono=emb,p:epi=surj}).

\begin{remark}[Compact ordered spaces as separated metric compact Hausdorff spaces]\label{d:rem:2}
  We recall from \cite{Nac65} that a \emph{compact ordered space} (also called \emph{compact pospace} or \emph{Nachbin space}) consists of a compact space $X$ together with a partial order $\leq$ on $X$ so that the set $\{(x,y)\in X\times X\mid x\leq y\}$ is closed in $X\times X$; such a space \(X\) is automatically Hausdorff.
  The category of compact ordered spaces and monotone and continuous maps is denoted by \(\PosComp\).

  Every compact ordered space \(X\) can be thought of as a separated metric compact Hausdorff space with the same compact Hausdorff space and the metric
  \begin{displaymath}
    d(x,y) \coloneqq
    \begin{cases}
      0 & \text{if \(x\leq y\)},\\
      \infty & \text{otherwise}.
    \end{cases}
  \end{displaymath}
  In fact, compact ordered spaces can be identified with the separated metric compact Hausdorff spaces whose metric takes values in $\{0, \infty\}$.
  Since every map between compact ordered spaces is monotone if and only if it is non-expansive with respect to the corresponding metrics, we obtain a fully faithful functor
  \begin{displaymath}
    \PosComp \longrightarrow\SMKH.
  \end{displaymath}
  This functor has a right adjoint which leaves maps unchanged and sends a separated metric compact Hausdorff space \(X\) (with metric \(d\)) to the compact ordered space \(X\) with the same topology and the partial order given by
  \begin{displaymath}
    x\le y\quad\text{whenever}\quad d(x,y)=0.
  \end{displaymath}
  In particular, we conclude that \(\PosComp\) is closed in \(\SMKH\) under colimits, and it is easy to see that \(\PosComp\) is closed in under limits, too.
  The unit interval \([0,1]\) is a cogenerator in \(\PosComp\) and \(\PosComp\) is well-powered; therefore, the Special Adjoint Functor Theorem guarantees that \(\PosComp\to\SMKH\) has also a left adjoint.
  \begin{displaymath}
    \begin{tikzcd}
      \PosComp
      \ar{r}[outer sep=0.5ex]{\bot}[swap,outer sep=0.5ex]{\bot}
      & \SMKH
      \ar[bend right]{l} \ar[bend left]{l}
    \end{tikzcd}
  \end{displaymath}
\end{remark}

\begin{remark}[Symmetrisation]
  \label{d:rem:1}
  Every metric space \(X\) with metric \(d\) can be symmetrised by putting $d_s(x,y) \coloneqq \max\{d(x,y), d(y,x)\}$ (see \cite{Law73}).
  This metric is compatible with the topology: if \(d \colon X\times X\to [0,\infty]\) is continuous, then so is \(d_s \colon X\times X\to [0,\infty]\), since the map
  \begin{displaymath}
    \max \colon [0,\infty]\times [0,\infty] \longrightarrow [0,\infty]
  \end{displaymath}
  is continuous with respect to the upper topology. In fact, this construction defines a right adjoint to the full embedding
  \begin{displaymath}
    \SSMKH \longrightarrow\SMKH
  \end{displaymath}
  of the category \(\SSMKH\) of symmetric separated metric compact Hausdorff spaces and continuous non-expansive maps into \(\SMKH\). Similarly to \cref{d:rem:2}, we conclude that the category $\SSMKH$ is closed under limits and colimits in $\SMKH$.
  Therefore the inclusion functor $\SSMKH \rightarrow \SMKH$ has also a left adjoint by the Adjoint Functor Theorem (the solution set condition is trivial).
\end{remark}

Below we briefly indicate an example where metric compact Hausdorff spaces played a crucial role.
\begin{example}
  It is well-known that every classical metric space with compact \emph{induced} topology is Cauchy complete. However, to infer Cauchy completeness of a metric space it suffices to show that there is a \emph{compatible} compact Hausdorff topology, that is, the metric space is part of a metric compact Hausdorff space. To give a trivial example, consider an arbitrary product of classical metric spaces with compact induced topology. The product metric does not in general induce a compact topology; however, this metric is Cauchy complete because the product topology is compact Hausdorff and compatible. Less trivially, for every metric space \((X,d)\), the space \(UX\) of all ultrafilters on \(X\) equipped with the metric
  \begin{displaymath}
    h(\mathcal{U},\mathcal{V})=\sup_{(A\in \mathcal{U},\,B\in \mathcal{V})}\inf_{(x\in A,\,y\in B)}d(x,y)
  \end{displaymath}
  is Cauchy complete because the Zariski topology on \(UX\) (which is independent of \(d\)) is compatible with the metric \(h\) (see \cite{HR18}).
\end{example}

In a metric compact Hausdorff space, the compatibility between the metric and the topology is weaker than the one when considering the topology induced by the metric.
Roughly speaking, in the latter case open balls are open and, consequently, closed balls are closed, whereby in a metric compact Hausdorff space closed balls are closed, but open balls need not be open. For example, take any compact Hausdorff space with the discrete metric, which assigns distance $1$ to each pair of distinct points. Every singleton is an ``open ball'' of radius $\frac{1}{2}$, but it might fail to be open.
The following result makes this relationship more precise.

\begin{proposition}\label{d:prop:1}
  Let \(d\) be a separated metric on a compact Hausdorff space \(X\). Then \(d\) symmetrically induces the topology of \(X\) if and only if the map $d \colon X \times X \longrightarrow [0,\infty]$  is continuous with respect to the lower topology on \([0,\infty]\).
  This condition, if $(X,d)$ is a separated metric compact Hausdorff space and $d$ does not attain the value $\infty$, is equivalent to the continuity of $d$ with respect to the Euclidean topology on $[0, \infty[$.
\end{proposition}
\begin{proof}
  If the topology on \(X\) is symmetrically induced by the metric \(d\),
  then, for every \(u\in [0,\infty]\), the set
  \begin{displaymath}
    \{(x,y)\in X\times X\mid d(x,y)<u\}
    =\bigcup_{z,\,u_1+u_2\leq u}B_l(z,u_1)\times B_r(z,u_2)
  \end{displaymath}
  is open in \(X\times X\). Hence, \(d \colon X\times X\to [0,\infty]\) is continuous with respect to the lower topology on \([0,\infty]\).
  Conversely, assume now that the metric is continuous with respect to the lower topology on \([0,\infty]\).
  Since \(d\) is separated, the topology induced by \(d\) is Hausdorff.
  Moreover, by continuity of \(d\), the sets
  \begin{displaymath}
    \{y\in X\mid d(x,y)<u\}\quad\text{and}\quad
    \{y\in X\mid d(y,x)<u\}
  \end{displaymath}
  are open also in the given compact Hausdorff topology, and hence the two topologies coincide.

  The last part of the statement follows from the fact that a map into $[0, \infty[$ is continuous with respect to the Euclidean topology if and only if it is continuous with respect to the lower and the upper topology.
\end{proof}

\begin{remark}
  \cref{d:prop:1} is a generalisation of the following well-known fact.
  If a partial order on a compact Hausdorff space \(X\) is open in \(X\times X\), then it is also closed and, moreover, the topology of \(X\) is generated by the sets $\downset x$ and $\upset x$,
  with \(x\in X\) (see also \cite[Theorem~5]{Nac65}).
\end{remark}


\section{Quotient objects and continuous submetrics}
\label{s:quotient-objects}

In this paper we will study equivalence relations in $\SMKHop$ by looking at their duals in $\SMKH$, which are particular epimorphisms.
In this section we prepare the ground by giving a description of surjective morphisms in \(\SMKH\) (= epimorphisms by \cref{p:epi=surj} below) which is similar to the presentation of surjections out of a set by equivalence relations.

We let $\Quot(X)$ denote the class of surjective morphisms of separated metric compact Hausdorff spaces with domain $X$.
We consider $\Quot(X)$ as the full subcategory of the \emph{coslice category} $X\downarrow \SMKH$ of objects of $\SMKH$ under $X$.
Since surjective morphisms are epimorphisms in $\SMKH$, the category $\Quot(X)$ is actually a preordered class:
\[
	f \leq g \quad \iff \quad
	\begin{tikzcd}
		&X \arrow[swap]{dl}{f}\arrow{rd}{g}\\
		Y \arrow[dashed,swap]{rr}{\exists h} && Z.
	\end{tikzcd}
\]

\begin{remark}
  \label{d:rem:3}
  Since the pair \((\{\text{surjective morphisms}\}, \{\text{embeddings}\})\) is an orthogonal factorisation system on $\SMKH$, $\Quot(X)$ is a coreflective full subcategory of $X\downarrow \SMKH$: the coreflector assigns to a morphism the surjective component in its factorization.
\end{remark}

\begin{definition}[Quotient object]\label{n:quotequiv}
  By identifying elements in the same equivalence class, from $\Quot(X)$ we obtain a partially ordered class (in fact, a set) $\QuotEquiv(X)$, whose elements we call \emph{quotient objects} of $X$.
  With a little abuse of notation, we take the liberty to refer to an element of $\QuotEquiv(X)$ just with one of its representatives.
\end{definition}

Our next goal is to encode quotient objects of $X$ internally on $X$.

\begin{remark}
	Encodings similar to those that we are seeking already appear for similar categories.
	For example, by \cite[The Alexandroff Theorem~3.2.11]{Engelking1989}\footnote{The reader is warned that, in \cite{Engelking1989}, by `compact space' is meant what we here call a compact \emph{Hausdorff} space.}, in the category $\CompHaus$ of compact Hausdorff spaces and continuous functions, a surjective morphism $f \colon X \epi Y$ is encoded by the equivalence relation ${\sim_f} \coloneqq \{(x,y) \in X \times X \mid f(x) = f(y)\}$; the equivalence relation $\sim_f$ is closed, and, in fact, there is a bijection between equivalence classes of surjective morphisms of compact Hausdorff spaces with domain $X$ and closed equivalence relations on $X$.
	There are also analogous versions for Boolean spaces, namely \emph{Boolean relations}\footnote{Sometimes called \emph{Boolean equivalences}.} \cite[Lemma~1, Chapter 37]{GivantHalmos2009}, for Priestley spaces, namely \emph{lattice preorders} \cite[Definition~2.3]{CignoliLafalceEtAl1991}\footnote{Lattice preorders are also called \emph{Priestley quasiorders} (\cite[Definition~3.5]{Schmid2002}), or \emph{compatible quasiorders}.}, and for Nachbin's compact ordered spaces \cite[Lemma~11]{AR20}.
\end{remark}

We encode a quotient object $f \colon X \epi Y$ of $X \in \SMKH$ via a certain (possibly non-separated) metric $\kappa_f$ on $X$, as follows.

\begin{definition}[Kernel metric]\label{n:les-f}
  For a morphism $f\colon X \to Y$ in $\SMKH$, we call \emph{kernel metric} of $f$ the initial metric structure (also called Cartesian lifting) on the compact Hausdorff space \(X\) with respect to the continuous map \(f\) and the metric on \(Y\); explicitly, the function $\kappa_f \colon X \times X \to [0, \infty]$ defined by
  \[
    \kappa_f(x_1,x_2) \coloneqq d_Y(f(x_1),f(x_2)).
  \]
\end{definition}

We note that, being the structure of a metric compact Hausdorff space, \(\kappa_f\) is continuous with respect to the upper topology on \([0,\infty]\).

\Cref{n:les-f} will be relevant especially for $f$ surjective.
The idea is that a surjective morphism $f$ can be completely recovered from $\kappa_f$ (up to an isomorphism). In order to establish an inverse for the assignment that maps $f$ to its kernel metric ${\kappa_f}$, we investigate the properties satisfied by kernel metrics: these properties are precisely being a continuous metric below the given metric (\cref{d:prop:2}).
\begin{definition}[Continuous submetric]
  Let \(X\) be a separated metric compact Hausdorff space with metric \(d\).
  A \emph{continuous submetric} \(\gamma\) on \(X\) is a (not necessarily separated) continuous metric \(\gamma \colon X\times X\to [0,\infty]\)  with respect to the upper topology of $[0,\infty]$ which is below the given metric \(d\), i.e., for all $x,y \in X$, $\gamma(x,y) \leq d(x,y)$.
\end{definition}
By construction, for every \(f \colon X\to Y\) in \(\SMKH\), \(\kappa_f\) is a submetric of the metric of \(X\).

\begin{remark}\label{r:sym}
  We shall now see how one recovers a surjective morphism $f$ from its kernel metric $\kappa_f$.
  Let $(X,d)$ be a separated metric compact Hausdorff space and let ${\gamma}$ be a continuous submetric on $X$.
  Then take the separation-reflection $X/{\sim_\gamma}$ of $X$ with respect to $\gamma$ (for instance, see \cite{HN20}).
  We recall that $X/{\sim_\gamma}$ is the quotient of $X$ with respect to the equivalence relation $\sim_\gamma$ defined by
  \[
  x \sim_\gamma y \iff \gamma(x,y) = \gamma(y,x) = 0;
  \]
  the topology is the quotient topology, and the metric is $d_{X/{\sim_\gamma}}([x],[y]) = \gamma(x,y)$.
  Since $\gamma$ is below $d$, the projection function $p_{\gamma}\colon X \epi X/{\sim_\gamma}$ is non-expansive: $d_{X/{\sim_\gamma}}([x],[y]) = \gamma(x,y) \leq d(x,y)$.
  Moreover, it is continuous since $X/{\sim_\gamma}$ has the quotient topology.
  Thus, the function $p_{\gamma}\colon X \epi X/{\sim_\gamma}$ is a surjective morphism in $\SMKH$.
\end{remark}

\begin{definition}
	For $X$ a separated metric compact Hausdorff space, we let $\Subm(X)$ denote the set of continuous submetrics on \(X\).
We equip $\Subm(X)$ with the pointwise partial order, i.e.\ $\gamma_1\leq \gamma_2$ whenever, for all $x,y\in X$, $\gamma_1(x,y) \leq \gamma_2(x,y)$.
\end{definition}

\begin{theorem}\label{d:prop:2}
  The assignments
  \begin{align*}
    \Quot(X) & \longrightarrow \Subm(X), & \Subm(X) & \longrightarrow	\Quot(X)\\
    \big(f \colon X\epi Y\big)	& \longmapsto {\kappa_f} & {\gamma} & \longmapsto \big(p_{\gamma}\colon X\epi X/{\sim_\gamma}\big)\notag
  \end{align*}
  form a dual equivalence. Therefore the posets \(\QuotEquiv(X)\) and \(\Subm(X)\) are dually isomorphic.
\end{theorem}
\begin{proof}
  First we observe that, for all \(f \colon X \epi Y\in\Quot(X)\) and \(\gamma\in\Subm(X)\), one has
  \begin{displaymath}
    \kappa_f\leq\gamma\iff
    \big(p_{\gamma}\colon X\epi X/{\sim_\gamma}\big)\leq f,
  \end{displaymath}
  that is, the two maps above form a dual adjunction. Clearly, for each \(\gamma\in\Subm(X)\), the kernel metric of \(p_{\gamma}\colon X\epi X/{\sim_\gamma}\) is \(\gamma\). For every $f \colon X \epi Y$, the canonical function $\eps_f \colon X/{{\sim_\gamma}} \to Y$ is an embedding because, for all $x,y \in X$, we have
  \[
    d([x],[y]) = \kappa_f(x,y) = d_Y(f(x),f(y)) =
    d_Y(\eps_f([x]),\eps_f([y])).
  \]
  Therefore, $\eps_f$ is an isomorphism.
\end{proof}


\section{Coregularity for separated metric compact Hausdorff spaces}
\label{sec:coreg-separ-metr}

The purpose of this section is to prove that the category $\SMKH$ is coregular, i.e., that $\SMKHop$ is regular.
Recall from \Cref{sec:prelimiaries} that one of the conditions of coregularity is that regular monomorphisms are preserved under pushouts.
For this reason, we start by describing pushouts of embeddings; later on we will prove that the embeddings are precisely the regular monomorphisms.

\begin{proposition}[Description of a pushout along an embedding] \label{p:itisametric}
	Let $i\colon A \rmono X$ be an embedding and $f \colon A \to B$ a morphism in $\SMKH$.
	Let $\iota_B \colon B \hookrightarrow B + X$ and $\iota_X \colon X \hookrightarrow B + X$ denote the coproduct maps.
	The following describes a continuous submetric $\gamma$ on $B + X$:
	\begin{enumerate}
		\item
		For $b,b' \in B$,
		\[
			\gamma(\iota_B(b),\iota_B(b')) = d_B(b,b').
		\]

		\item
		For $x,x' \in X$,
		\begin{align*}
			&\gamma(\iota_X(x),\iota_X(x'))\\
			&= \min\mleft\{d_{X}(x,x'),\,\inf_{a, a' \in A}\big(d_X(x,i(a)) + d_{B}(f(a), f(a')) + d_{X}(i(a'),x')\big)\mright\}.
		\end{align*}

		\item \label{i:mixed}
		For $b \in B$ and $x \in X$,
		\begin{align*}
			\gamma(\iota_B(b),\iota_X(x)) & = \inf_{a \in A}\big(d_B(b,f(a)) + d_X(i(a),x) \big),\\
			\gamma(\iota_X(x),\iota_X(x)) & = \inf_{a \in A}\big(d_X(x,i(a)) + d_{B}(f(a), b)\big).
		\end{align*}
	\end{enumerate}
	Moreover, the quotient $(B+X)/{\sim_\gamma}$ (with the obvious maps) is the pushout of $B \leftarrow A \hookrightarrow X$.
\end{proposition}

\begin{proof}

	For $x \in X$, $b \in B$ and $a \in A$, we write $x$ for $\iota_X(x)$, $b$ for $\iota_B(b)$, and $a$ for $i(a)$ and $\iota_X(i(a))$.

	A case analysis shows immediately that $\gamma$ is below $d_{B+X}$.
	To prove continuity of $\gamma \colon (B + X) \times (B + X) \to [0, \infty]$ with respect to the upper topology of $[0,\infty]$, it is enough to prove that it is continuous over each of its four pieces $B \times B$, $B \times X$, $X \times B$ and $X \times X$.
	It is clear that $\gamma$ is continuous over $B \times B$.
	We show continuity of $B \times X$; the other cases are similar.
	The function
	\begin{align*}
		B \times X & \longrightarrow [0, \infty]\\
		(b, x) & \longmapsto \inf_{a \in A} \big( d_B(b, f(a)) + d_X(a, x) \big)
	\end{align*}
	can be written as the composite of the following two functions
	\begin{align*}
		B \times X & \longrightarrow [0, \infty]^A & \inf \colon [0, \infty]^A& \longrightarrow [0, \infty]\\
		(b, x) & \longmapsto \big(a \mapsto d_B(b, f(a)) + d_X(a, x)\big) & f & \longmapsto \inf_{a \in A} f(a),
	\end{align*}
	where $[0, \infty]^A$ is the exponential in the category of topological spaces, which is the set of continuous functions from $A$ to $[0, \infty]$ equipped with the compact-open topology, and which exists because $A$ is a compact Hausdorff space (see \cite{EH01}, for instance). The first function is continuous because it is the transpose of the function
	\begin{align*}
		A \times B \times X & \longrightarrow [0, \infty]\\
		(a,b, x) & \longmapsto  d_B(b, f(a)) + d_X(a, x),
	\end{align*}
	which is continuous because $d_B$, $d_X$, $f$, and $+ \colon [0, \infty]^2 \to [0, \infty]$ are continuous.
	The function $\inf \colon [0, \infty]^A \to [0, \infty]$ is continuous because, for every \(u\in [0,\infty]\), we have
        \begin{align*}
          {\inf}^{-1}({]u,\infty]}) & =  \{\varphi \in [0,\infty]^A \mid \inf \varphi > u\} && \text{by definition of $\inf$}\\
          & = \{\varphi \in [0,\infty]^A \mid \exists v > u \text{ s.t.\ } \inf \varphi \geq v\} && \text{since $\varphi[A]$ is compact}\\
          &=\bigcup_{v>u}\{\varphi \in [0,\infty]^A \mid \varphi[A] \subseteq {]v,\infty]} \},
        \end{align*}
        and this set is open in \([0,\infty]^A\).
	Hence, their composite $B \times X \to [0, \infty]$ is continuous, as desired.

	Let us now prove that $\gamma$ is a metric.

	It is immediate that for all $z \in B + X$ we have $\gamma(z,z) = 0$.

	We now prove by cases that $\gamma$ satisfies the triangle inequality.

	Clearly, for all $b, b', b'' \in B$, we have
	\[
	\gamma(b, b'') = d_B(b, b'') \leq d_B(b, b') + d_B(b', b'') = \gamma(b, b') + \gamma(b', b''),
	\]
	proving a case of the triangle inequality.

	Let $x, x', x'' \in X$, and let us prove $\gamma(x, x'') \leq \gamma(x, x') + \gamma(x', x'')$.
	We have
	\begin{align*}
		\gamma(x,x'')& = \min\mleft\{d_{X}(x,x'),\,\inf_{a, a' \in A}\big(d_X(x,a) + d_{B}(f(a), f(a')) + d_{X}(a',x'') \big)\mright\}\\
		& \leq d_{X}(x,x'')\\
		& \leq d_{X}(x,x') + d_{X}(x',x'').
	\end{align*}
	Moreover, for all $a_0, a_0' \in A$, we have
	\begin{align*}
		\gamma(x,x'') & = \min\mleft\{d_{X}(x,x'),\,\inf_{a, a' \in A}\big(d_X(x,a) + d_{B}(f(a), f(a')) + d_{X}(a',x'')\big)\mright\}\\
		& \leq d_X(x,a_0) + d_{B}(f(a_0), f(a_0')) + d_{X}(a_0',x'')\\
		& \leq d_X(x,a_0) + d_{B}(f(a_0), f(a_0')) + d_{X}(a_0',x') + d_{X}(x',x'').
	\end{align*}
	Moreover, for all $a_1, a_1' \in A$, we have
	\begin{align*}
		\gamma(x,x'') & = \min\mleft\{d_{X}(x,x'),\,\inf_{a, a' \in A}\big(d_X(x,a) + d_{B}(f(a), f(a')) + d_{X}(a',x'')\big)\mright\}\\
		& \leq d_X(x,a_1) + d_{B}(f(a_1), f(a_1')) + d_{X}(a_1',x'')\\
		& \leq d_{X}(x,x') + d_X(x',a_1) + d_{B}(f(a_1), f(a_1')) + d_{X}(a_1',x'').
	\end{align*}
	Moreover, for all $a_0, a_0', a_1, a_1' \in A$ we have
	\begin{align*}
		&\gamma(x,x'') \\
		& = \min\mleft\{d_{X}(x,x'),\,\inf_{a, a' \in A}\big(d_X(x,a) + d_{B}(f(a), f(a')) + d_{X}(a',x'')\big)\mright\}\\
		& \leq d_X(x,a_0) + d_{B}(f(a_0), f(a_1')) + d_{X}(a_1',x'')\\
		& \leq d_X(x,a_0) + d_{B}(f(a_0), f(a_0')) + d_{B}(f(a_0'),f(a_1)) + d_{B}(f(a_1), f(a_1')) + d_{X}(a_1',x'')\\
		& \leq d_X(x,a_0) + d_{B}(f(a_0), f(a_0')) + d_{X}(a_0',a_1) + d_{B}(f(a_1), f(a_1')) + d_{X}(a_1',x'')\\
		& \leq d_X(x,a_0) + d_{B}(f(a_0), f(a_0')) + d_{X}(a_0',x') + d_X(x',a_1) + d_{B}(f(a_1), f(a_1')) + d_{X}(a_1',x'').
	\end{align*}
	Therefore, putting the four previous displays together, we obtain
	\begin{align*}
		&\gamma(x,x'')\\
		& \leq\min\Big\{d_{X}(x,x') + d_{X}(x',x''),\\
		& \inf_{a_0, a_0' \in A}\big(d_X(x,a_0) + d_{B}(f(a_0), f(a_0')) + d_{X}(a_0',x') + d_{X}(x',x'')\big),\\
		& \inf_{a_1, a_1' \in A}\big(d_{X}(x,x') + d_X(x',a_1) + d_{B}(f(a_1), f(a_1')) + d_{X}(a_1',x'')\big),\\
		& \inf_{a_0,a_0', a_1,a_1' \in A}\big(d_X(x,a_0) + d_{B}(f(a_0), f(a_0')) + d_{X}(a_0',x') + d_X(x',a_1)+\\
                & \qquad\qquad  d_{B}(f(a_1), f(a_1')) + d_{X}(a_1',x'') \big)\Big\}\\
		& = \min\mleft\{d_{X}(x,x'),\,\inf_{a_0, a_0' \in A}\big(d_X(x,a_0) + d_{B}(f(a_0), f(a_0')) + d_{X}(a_0',x') \big)\mright\}\\
		& \quad + \min\mleft\{d_{X}(x',x''),\,\inf_{a_1, a_1' \in A}\big(d_X(x',a_1) + d_{B}(f(a_1), f(a_1')) + d_{X}(a_1',x'')\big)\mright\}\\
		& =\gamma(x, x') + \gamma(x', x''),
	\end{align*}
	proving another case of the triangle inequality.

	Let $x \in X$ and $b, b' \in B$, and let us prove $\gamma(x,b') \leq \gamma(x,b) + \gamma(b, b')$.
	We have
	\begin{align*}
		\gamma(x,b') & = \inf_{a \in A}\big(d_X(x,a) + d_{B}(f(a), b')\big)\\
		& \leq \inf_{a \in A}\big(d_X(x,a) + d_{B}(f(a), b) + d_B(b,b')\big)\\
		& = \mleft(\inf_{a \in A}\big(d_X(x,a) + d_{B}(f(a), b)\big)\mright) + d_B(b,b')\\
		& = \gamma(x,b) + \gamma(b, b').
	\end{align*}
	This proves another case of the triangle inequality.
	Similarly, one proves $\gamma(b',x) \leq \gamma(b', b) + \gamma(b,x)$.

	Now, let $x,x' \in X$ and $b \in B$, and let us prove $\gamma(x,b) \leq \gamma(x,x') + \gamma(x', b)$.
	For all $a_1 \in A$ we have
	\begin{align*}
		\gamma(x,b) & = \inf_{a \in A}\big(d_X(x,a) + d_{B}(f(a), b)\big)\\
		&\leq d_X(x,a_1) + d_{B}(f(a_1), b)\\
		& \leq d_{X}(x,x') +d_X(x',a_1) + d_{B}(f(a), b).
	\end{align*}
	Moreover, for all $a_0, a_0', a_1 \in A$, we have
	\begin{align*}
		\gamma(x,b) & = \inf_{a \in A}\big(d_X(x,a) + d_{B}(f(a), b)\big)\\
		& \leq d_X(x,a_0) + d_{B}(f(a_0), b)\\
		& \leq d_X(x,a_0) + d_{B}(f(a_0), f(a_0')) + d_{B}(f(a_0'),f(a_1)) + d_{B}(f(a_1), b)\\
		& \leq d_X(x,a_0) + d_{B}(f(a_0), f(a_0')) + d_{A}(a_0',a_1) + d_{B}(f(a_1), b)\\
		& = d_X(x,a_0) + d_{B}(f(a_0), f(a_0')) + d_{X}(a_0',a_1) + d_{B}(f(a_1), b)\\
		& \leq d_X(x,a_0) + d_{B}(f(a_0), f(a_0')) + d_{X}(a_0',x') + d_X(x',a_1) + d_{B}(f(a_1), b).
	\end{align*}
	By combining the two displayed inequalities above, we get
	\begin{align*}
		& \gamma(x,b) \\
		& \leq \min\Big\{\inf_{a_1 \in A}\big(d_{X}(x,x') +d_X(x',a_1) + d_{B}(f(a_1), b)\big),\\
		& \inf_{a_0, a_0', a_1 \in A}\big(d_X(x,a_0) + d_{B}(f(a_0), f(a_0')) + d_{X}(a_0',x') + d_X(x',a_1) + d_{B}(f(a_1), b) \big)\Big\}\\
		& = \min\mleft\{d_{X}(x,x'),\,\inf_{a_0, a_0' \in A}\big(d_X(x,a_0) + d_{B}(f(a_0), f(a_0')) + d_{X}(a_0',x')\big)\mright\}\\
		& \quad +  \inf_{a_1 \in A}\big(d_X(x',a_1) + d_{B}(f(a_1), b)\big)\\
		& = \gamma(x,x') + \gamma(x', b).
	\end{align*}
	This proves another case of the triangle inequality.
	Similarly, one proves $\gamma(b, x) \leq \gamma(b, x') + \gamma(x', x)$.

	Now, let $x, x' \in X$ and $b \in B$, and let us prove $\gamma(x, x') \leq \gamma(x, b) + \gamma(b, x')$.
	For all $a_0, a_0' \in A$ we have
	\begin{align*}
		\gamma(x, x') & = \min\mleft\{d_{X}(x,x'),\,\inf_{a, a' \in A}\big(d_X(x,a) + d_{B}(f(a), f(a')) + d_{X}(a',x')\big)\mright\}\\
		& \leq d_X(x, a_0) + d_B(f(a_0), f(a_0')) + d_X(a_0', x') \\
		& \leq d_X(x, a_0) + d_B(f(a_0), b) + d_B(b, f(a_0')) + d_X(a_0', x').
	\end{align*}
	Therefore,
	\begin{align*}
		\gamma(x, x') & \leq \inf_{a_0, a_0' \in A}\big(d_X(x, a_0) + d_B(f(a_0), b) + d_B(b, f(a_0')) + d_X(a_0', x')\big)\\
		& = \inf_{a_0 \in A}\big(d_X(x,a_0) + d_{B}(f(a_0), b)\big) + \inf_{a_0' \in A}\big(d_B(b,f(a_0')) + d_X(a_0',x')\big)\\
		& = \gamma(x, b) + \gamma(b, x').
	\end{align*}
	This proves another case of the triangle inequality.

	Finally, let $x \in X$ and $b, b' \in B$, and let us prove $\gamma(b, b') \leq \gamma(b, x) + \gamma(x, b')$.
	For all $a_0, a_0' \in A$ we have
	\begin{align*}
		\gamma(b, b')& = d_B(b, b')\\
		& \leq d_B(b,f(a_0)) + d_B(f(a_0),f(a_0')) + d_{B}(f(a_0'), b')\\
		& \leq d_B(b,f(a_0)) + d_A(a_0,a_0') + d_{B}(f(a_0'), b')\\
		& = d_B(b,f(a_0)) + d_X(a_0,a_0') + d_{B}(f(a_0'), b')\\
		& \leq d_B(b,f(a_0)) + d_X(a_0,x) + d_X(x,a_0') + d_{B}(f(a_0'), b').
	\end{align*}
	Therefore,
	\begin{align*}
		\gamma(b, b') & \leq \inf_{a_0, a_0' \in A}\big(d_B(b,f(a_0)) + d_X(a_0,x) + d_X(x,a_0') + d_{B}(f(a_0'), b')\big)\\
		& = \inf_{a_0 \in A}(d_B(b,f(a_0)) + d_X(a_0,x)) + \inf_{a_0' \in A}\big(d_X(x,a_0') + d_{B}(f(a_0), b') \big)\\
		& = \gamma(b, x) + \gamma(x, b').
	\end{align*}
	This proves the last case of the triangle inequality.

	This proves that $\gamma$ is a continuous submetric on $B + X$.

	We now have a diagram
	\[
		\begin{tikzcd}
		      A \arrow[hookrightarrow]{r}{i} \arrow[swap]{d}{f} & X	\arrow{d}{} \\
		      B \arrow[swap]{r}{} & (B + X)/{\sim}. \arrow[ul, phantom, very near start]
	    \end{tikzcd}
	\]
	We prove that it is a pushout.
	The square commutes since, by \eqref{i:mixed}, for all $a \in A$
	\[
		\gamma(f(a), a) = \gamma(a,f(a)) = 0.
	\]
	Let us prove the universal property.
	Let
	\[
		\begin{tikzcd}
		      A \arrow[hookrightarrow]{r}{i} \arrow[swap]{d}{f} & X	\arrow{d}{g} \\
		      B \arrow[swap]{r}{j} & Y \arrow[ul, phantom, very near start]
	    \end{tikzcd}
	\]
	be a commutative square.
	Let $\gamma'$ be the continuous submetric on $B + X$ associated to the canonical morphism $B + X \to Y$.
	By \cref{d:prop:2} and \cref{d:rem:3}, it is enough to prove $\gamma \geq \gamma'$.

	Since $j$ is non-expansive, for all $b,b' \in B$ we have
	\[
		\gamma'(b,b') = d_{Y}(j(b),j(b')) \leq d_B(b,b') = \gamma(b,b').
	\]

	Let $x,x' \in X$. Since $g$ is non-expansive, we have
	\[
		d_{Y}(g(x),g(x')) \leq d_{X}(x, x').
	\]
	Moreover, for all $a, a' \in A$, we have
	\begin{align*}
    		&d_Y(g(x), g(x'))\\
		& \leq d_Y(g(x), jf(a)) + d_Y(jf(a), jf(a')) + d_Y(jf(a'),g(x')) && \text{(by the triangle ineq.)}\\
		&  = d_Y(g(x), g(a)) + d_Y(jf(a), jf(a')) + d_Y(g(a'),g(x')) && \text{(since $j \circ f = g \circ i$)}\\
		& \leq d_X(x, a) + d_B(f(a),f(b)) + d_{X}(a',x'),
    \end{align*}
    where the last inequality holds since $g$ and $j$ are non-expansive.
    Thus,
	\begin{align*}
		\gamma'(x,x') & = d_Y(g(x),g(x'))\\
		&\leq \min\mleft\{d_{X}(x,x'),\,\inf_{a, a' \in A}\big(d_X(x,a) + d_{B}(f(a), f(a')) + d_{X}(a',x')\big)\mright\}\\
		& = \gamma(x,x').
	\end{align*}

    Let $b \in B$ and $x \in X$.
    For all $a \in A$ we have
    \begin{align*}
		& d_Y(j(b),g(x))\\
		& \leq d_Y(j(b),jf(a)) + d_Y(g(a),g(x)) && \text{(by triangle ineq.\ and since $j \circ f = g \circ i$)}\\
		& \leq d_B(b,f(a)) + d_X(a,x),
    \end{align*}
    and, similarly,
    \begin{align*}
    		& d_Y(g(x),j(b))\\
		& \leq d_Y(g(x),g(a)) + d_Y(jf(a),j(b))  && \text{(by triangle ineq.\ and since $g \circ i = j \circ f$)}\\
		& \leq d_X(x,(a)) + d_B(f(a),b),
    \end{align*}
    Therefore,
    \begin{align*}
    		\gamma'(b,x) = d_{Y}(j(b),g(x)) & = \inf_{a \in A}\big(d_B(b,f(a)) + d_X(a,x)\big) = \gamma(b,x),\\
		\gamma'(x,b) = d_{Y}(g(x),j(b)) &= \inf_{a \in A}\big(d_X(x,a) + d_{B}(f(a), b)\big) = \gamma(x,b). \qedhere
    \end{align*}
\end{proof}

\begin{corollary}\label{c:pushout-of-emb-is-emb}
	In $\SMKH$, the pushout of an embedding along any morphism is an embedding.
\end{corollary}

We specialise \cref{p:itisametric} to the case of a pushout of two embeddings.

\begin{corollary}[of \cref{p:itisametric}]\label{l:pushout-mono}
  Consider embeddings $f_0\colon X\rmono Y_0$ and $f_1\colon X\rmono Y_1$ in $\SMKH$ and their pushout as displayed below.
	\[
	\begin{tikzcd}
		X \arrow[hookrightarrow]{r}{f_1}	\arrow[hookrightarrow,swap]{d}{f_0}										& Y_1	\arrow[hook]{d}{\lambda_1} \\
		Y_0  \arrow[swap, hook]{r}{\lambda_0}	& P		\arrow[ul, phantom, "\ulcorner", very near start]
	\end{tikzcd}
	\]
	For all $i,j\in \{0,1\}$, $u \in Y_i$ and $v \in Y_j$,
	\[
	d_P(\lambda_i(u), \lambda_j(v)) = \begin{cases}
		d_{Y_i}(u,v) & \text{if }i = j,\\
		\inf_{x \in X}\big(d_{Y_i}(u,f_i(x)) + d_{Y_{j}}(f_{j}(x),v)\big) & \text{if }i \neq j.
	\end{cases}
	\]
\end{corollary}

It is known that, in a regular category, any pullback square consisting entirely of regular epimorphisms is also a pushout square. This follows from (the dual of) the main result of \cite{Ringel1972} (cf.\ also \cite[§1.565]{FreydScedrov1990} or \cite[Remark 5.3]{CarboniKellyEtAl1993}).
Since we are aiming to show that $\SMKH$ is coregular (and that embeddings coincide with regular monomorphisms), the following result should not come as a surprise.

\begin{lemma} \label{l:pushout-of-emb-is-pullback}
	A pushout square in $\SMKH$ consisting entirely of embeddings is also a pullback.
\end{lemma}

\begin{proof}
	Let us give names to the morphisms in the pushout square.
	\[
	\begin{tikzcd}
		X \arrow[hookrightarrow]{r}{f_1}	  \arrow[hookrightarrow,swap]{d}{f_0}& Y_1	\arrow[hookrightarrow]{d}{\lambda_1} \\
		Y_0 \arrow[swap,hookrightarrow]{r}{\lambda_0} & P		\arrow[ul, phantom, "\ulcorner", very near start]
	\end{tikzcd}
	\]
	Assume we have $y_0 \in Y_0$ and $y_1 \in Y_1$ with $\lambda_0(y_0) = \lambda_1(y_1)$. We prove that there is (a necessarily unique) $x \in X$ such that $y_0 = f_0(x)$ and $y_1 = f_1(x)$.
	By \cref{l:pushout-mono},
	\[
		0 = d_P(\lambda_0(y_0), \lambda_1(y_1)) = \inf_{x \in X} \big(d_{Y_0}(y_0, f_0(x)) + d_{Y_1}(f_1(x),y_1) \big).
	\]
	The set
	\[
		C \coloneqq \{d_{Y_0}(y_0, f_0(x)) + d_{Y_1}(f_1(x), y_1) \mid x \in X\}
	\]
	is compact in $[0,\infty]$ with respect to the upper topology because it is the image of the compact space $X$ under the continuous function
	\begin{align*}
		X & \longrightarrow [0, \infty]\\
		x & \longmapsto d_{Y_0}(y_0, f_0(x)) + d_{Y_1}(f_1(x), y_1).
	\end{align*}
	Since $C$ is compact and $\inf C = 0$, we have $0 \in C$, and hence there is $x \in X$ such that $d_{Y_0}(y_0, f_0(x)) + d_{Y_1}(f_1(x), y_1) = 0$, and hence $d_{Y_0}(y_0, f_0(x)) = d_{Y_1}(f_1(x), y_1) = 0$.
	Similarly, there is $x' \in X$ such that $d_{Y_1}(y_1, f_1(x')) = d_{Y_0}(f_0(x'), y_0) = 0$.
	Therefore,
	\[
		d_{Y_0}(f_0(x'), f_0(x)) \leq  d_{Y_0}(f_0(x'), y_0) + d_{Y_0}(y_0, f_0(x)) = 0,
	\]
	and hence $d_X(x', x) = 0$ since $f_0$ is an embedding.
	Similarly, one shows $d_X(x, x') = 0$.
	Therefore, since $X$ is separated, $x = x'$.
	Thus, we have also $d_{Y_1}(y_1, f_1(x)) = 0$ and $d_{Y_0} (f_0(x), y_0) = 0$, and by separation we get $y_0 = f_0(x)$ and $y_1 = f_1(x)$.
	So, the square is a pullback in $\Set$, and in $\SMKH$, too, since all maps are embeddings.
\end{proof}

\begin{proposition} \label{p:rmono=emb}
	A morphism in $\SMKH$ is a regular monomorphism if and only if it is an embedding.
\end{proposition}

\begin{proof}
	A regular mono in $\SMKH$ is in particular a regular mono of compact Hausdorff spaces and of metric spaces (since the forgetful functors are right adjoint and hence preserve limits), which implies that it is an embedding.

	Conversely, let $i \colon X \hookrightarrow Y$ be an embedding.
	By \cref{c:pushout-of-emb-is-emb}, the pushout of $i$ along itself consists entirely of embeddings.
	\[
	\begin{tikzcd}
		X \arrow[hookrightarrow]{r}{i}	 \arrow[hookrightarrow,swap]{d}{i}& Y_1	\arrow[hookrightarrow]{d}{\lambda_1} \\
		Y_0\arrow[swap,hookrightarrow]{r}{\lambda_0} & P		\arrow[ul, phantom, "\ulcorner", very near start]
	\end{tikzcd}
	\]
	Thus, by \cref{l:pushout-of-emb-is-pullback} it is also a pullback.
	Hence, $i$ is the equaliser of $\lambda_0$ and $\lambda_1$.
\end{proof}

\begin{proposition} \label{p:epi=surj}
	A morphism in $\SMKH$ is an epimorphism if and only if it is surjective.
\end{proposition}

\begin{proof}
  Clearly, every surjection is an epimorphism. For the converse direction, let $f$ be an epimorphism, and consider its (surjection, embedding)-factorisation (\cref{d:thm:2}) $f = i \circ q$.
  By \cref{p:rmono=emb}, $i$ is a regular mono. Since $f$ is an epi, $i$ is also an epi, and therefore an isomorphism. Consequently, $f$ is surjective.
\end{proof}

The following is an immediate consequence of \cref{c:pushout-of-emb-is-emb,p:rmono=emb}.
\begin{corollary} \label{c:pushout-reg-mono-is-reg-mono} In $\SMKH$, the pushout of a regular monomorphism along any morphism is a regular monomorphism.
\end{corollary}

Combining \cref{d:thm:1,d:thm:2,p:epi=surj,p:rmono=emb,c:pushout-reg-mono-is-reg-mono}, we finally obtain the main result of the section.

\black

\begin{theorem} \label{t:regular}
	$\SMKHop$ is a regular category.
\end{theorem}

\section{Barr-coexactness for separated metric compact Hausdorff spaces}
\label{sec:barr-coex-separ}

In this section we show that the dual of the category \(\SMKH\) is Barr-exact.

We refer to \cite{Bor94a} for the notion of an internal (equivalence) relation. We recall that an equivalence relation is \emph{effective} if it is a kernel pair. Equivalently (in a regular category with coequaliser) if it is the kernel pair of its coequaliser.
Moreover, a \emph{Barr-exact category} is a regular category where every equivalence relation is effective.

In this section we provide a description of equivalence relations in the category $\SMKHop$, and we exploit it to prove that $\SMKHop$ is Barr-exact.

\begin{notation}
	Given morphisms $f_0\colon X \to Y_0$ and $f_1\colon X \to Y_1$, the unique morphism induced by the universal property of the product is denoted by $\langle f_0, f_1\rangle\colon X\to Y_0\times Y_1$.
	Similarly, given morphisms $f_0\colon X_0 \to Y$ and $f_1\colon X_1 \to Y$, the coproduct map is denoted by $\binom{g_0}{g_1}\colon X_0+X_1\to Y$.
\end{notation}

Dualising the definition of an internal binary relation, given a separated metric compact Hausdorff space $X$, we call a \emph{binary corelation} on $X$ a quotient object $\binom{q_0}{q_1}\colon X+ X \epi S$ of the separated metric compact Hausdorff space $X+X$.
We recall from \cref{d:thm:1} that $X + X$ is the disjoint union of two copies of $X$ equipped with the coproduct topology and metric.
A binary corelation on $X$ is called respectively \emph{reflexive}, \emph{symmetric}, \emph{transitive} provided it satisfies the properties:\\
\begin{minipage}[t]{0.49\textwidth}
	\begin{figure}[H]
		\centering
		\begin{tikzcd}
			{} 							& X+X \arrow[two heads,swap]{dl}{\binom{q_0}{q_1}} \arrow[two heads]{dr}{\binom{1_X}{1_X}}	& \\
			S\arrow[dashed]{rr}{\exists d} 	& 																															& X
		\end{tikzcd}
		{\vspace{-5pt}\caption*{reflexivity}}
	\end{figure}
\end{minipage}
\begin{minipage}[t]{0.49\textwidth}
	\begin{figure}[H]
		\centering
		\begin{tikzcd}
			& X+X \arrow[swap, two heads]{dl}{\binom{q_0}{q_1}} \arrow[two heads]{dr}{\binom{q_1}{q_0}}	& \\
			S \arrow[dashed]{rr}{\exists s}	& 																															& S
		\end{tikzcd}
		{\vspace{-5pt}\caption*{symmetry}}
	\end{figure}
\end{minipage}
\begin{figure}[H]
	\centering
	\begin{tikzcd}
		X \arrow{r}{q_1} \arrow[swap]{d}{q_0}	& S \arrow{d}{\lambda_1} \\
		S \arrow[swap]{r}{\lambda_0}			 	& P \arrow[ul, phantom, "\ulcorner", very near start]
	\end{tikzcd}
	\ \ \ \ $\Longrightarrow$ \ \ \ \
	\begin{tikzcd}
		& X+X \arrow[swap, two heads]{dl}{\binom{q_0}{q_1}} \arrow{dr}{\binom{\lambda_0\circ q_0}{\lambda_1\circ q_1}}	& \\
		S\arrow[dashed]{rr}{\exists t}	& 																																					& P
	\end{tikzcd}
	{\vspace{-5pt}\caption*{transitivity}}
\end{figure}%
\noindent An \emph{equivalence corelation} is a reflexive symmetric transitive binary corelation.
The key observation is that, since quotient objects of $X + X$ are in bijection with certain metrics on $X + X$, equivalence corelations are more manageable than their duals.

\begin{definition}
	Let $X$ be a separated metric compact Hausdorff space.
	A \emph{binary continuous submetric} on $X$ is an element of $\Subm(X + X)$, i.e.\ a continuous metric on $X+X$ which is below the coproduct metric on $X+X$.
	Furthermore, we say it is \emph{reflexive} (resp.\ \emph{symmetric}, \emph{transitive}, \emph{equivalence}) if the corresponding binary corelation on $X$ (under the correspondence with $\QuotEquiv(X+X)$ given by \Cref{d:prop:2}) is reflexive (resp.\ symmetric, transitive, equivalence).
\end{definition}

\begin{notation}
	We denote the elements of $X+X$ by $(x,i)$, where $x$ varies in $X$ and $i$ varies in $\{0,1\}$.
	Furthermore, $i^*$ stands for $1-i$; for example, $(x,1^*)=(x,0)$.
\end{notation}

As we will prove, every equivalence continuous submetric $\gamma$ on a separated metric compact Hausdorff space $X$ is obtained as follows: consider a closed subset $Y$ of $X$ and let $\gamma$ be the greatest metric on $X + X$ that extends the coproduct metric of $X + X$ and that satisfies $d((y,0), (y,1)) = \gamma((y,1),(y,0)) = 0$ for every $y \in Y$.

\begin{lemma}\label{l:refl}
	A binary continuous submetric $\gamma$ on a separated metric compact Hausdorff space $X$ is reflexive if and only if, for all $x, y \in X$ and $i, j \in \{0,1\}$,
	\begin{equation*}
		d_X(x, y) \leq \gamma((x,i), (y,j)).
	\end{equation*}
\end{lemma}

\begin{proof}
	Let $\binom{q_0}{q_1}\colon X+X \epi S$ be the binary corelation associated with $\gamma$.
	By the definition of a reflexive continuous submetric, $\gamma$ is reflexive if and only if $\binom{q_0}{q_1}\colon X+X \epi S$ is below $\binom{1_X}{1_X}\colon X+X \epi X$ in the poset $\QuotEquiv(X+X)$.
	By \cref{d:prop:2}, this is equivalent to ${\kappa_{\binom{1_X}{1_X}}} \leq \gamma$.
	Then, it suffices to note that, for all $(x,i),(y,j)\in X+X$,
	\[
	\kappa_{\binom{1_X}{1_X}} ((x,i),(y,j)) = d(x, y). \qedhere
	\]
\end{proof}

\begin{remark} \label{r:reflexive-is-emb}
	Note that any reflexive continuous submetric $\gamma$ on $X$ satisfies, for all $x, y\in X$ and $i \in \{0,1\}$,
	\begin{equation*} \label{e:refl}
		\gamma((x, i), (y, i)) = d(x, y).
	\end{equation*}
	Indeed, the inequality $\geq$ follows from \cref{l:refl}, while the inequality $\leq$ holds because $\gamma$ is below the coproduct metric of $X+X$.
\end{remark}

\begin{lemma}\label{l:symm}
	A binary continuous submetric $\gamma$ on a separated metric compact Hausdorff space $X$ is symmetric if and only if, for all $x, y \in X$ and $i, j \in \{0,1\}$,
	\begin{equation*}
		\gamma((x,i), (y,j)) = \gamma((x,i^*),(y,j^*)).
	\end{equation*}
\end{lemma}
\begin{proof}
	Let $\binom{q_0}{q_1}\colon X+X \epi S$ be the binary corelation associated with $\gamma$.
	By the definition of a symmetric continuous submetric, $\gamma$ is symmetric if and only if $\binom{q_0}{q_1}\colon X+X \epi S$ is below $\binom{q_1}{q_0}\colon X+X \epi S$ in $\QuotEquiv(X+X)$.
	By \cref{d:prop:2}, this happens exactly when ${\kappa_{\binom{q_1}{q_0}}} \leq \gamma$.
	For all $(x,i),(y,j)\in X+X$, we have
	\[
	\kappa_{\binom{q_1}{q_0}}((x,i),(y,j)) = \gamma((x,i^*), (y,j^*)).
	\]
	Therefore, the binary continuous submetric $\gamma$ is symmetric if and only if, for all $x, y \in X$ and $i, j \in \{0,1\}$, $\gamma((x,i),(y,j)) \leq \gamma((x,i^*), (y,j^*))$.
	Now use the fact that the statement with $\leq$ is equivalent to the statement with $=$ (since $i^{**} = i$).
\end{proof}

\begin{lemma}\label{l:transitive}
	A reflexive continuous submetric $\gamma$ on a separated metric compact Hausdorff space $X$ is transitive if and only if, for all $x,y \in X$ and $i \in \{0,1\}$,
	\[
	\gamma((x,i),(y,i^*)) = \inf_{z \in X} \big(\gamma((x,i), (z,i^*)) + \gamma((z,i),(y,i^*))\big).
	\]
\end{lemma}

\begin{proof}
	Let $\binom{q_0}{q_1}\colon X+X \epi S$ be the binary corelation associated with $\gamma$.
	To improve readability, we write $[x,i]$ instead of $\binom{q_0}{q_1}(x,i)$.
	By definition of transitivity, the binary continuous submetric $\gamma$ is transitive if and only if, given a pushout square in $\SMKH$ as in the left-hand diagram below,
	there is a morphism $t\colon S\to P$ such that the right-hand diagram commutes.
	\[
	\begin{tikzcd}
		X \arrow[swap]{d}{q_1} \arrow{r}{q_0} & S \arrow{d}{\lambda_1} \\
		S \arrow[swap]{r}{\lambda_0}				& P \arrow[ul, phantom, "\ulcorner", very near start]
	\end{tikzcd}
	\ \ \ \ \ \ \ \
	\begin{tikzcd}
		{} 	& X+X \arrow[swap, two heads]{dl}{\binom{q_0}{q_1}} \arrow{dr}{\binom{\lambda_0\circ q_0}{\lambda_1\circ q_1}}	& \\
		S \arrow[dashed]{rr}{t}	& 		& P
	\end{tikzcd}
	\]
	By \cref{d:prop:2} and \cref{d:rem:3}, such a $t$ exists precisely when
	\[
	\kappa_{\binom{\lambda_0\circ q_0}{\lambda_1\circ q_1}} \leq \kappa_{\binom{q_0}{q_1}},
	\]
	i.e., when, for every $(x,i),(y,j)\in X+X$,
	\[
	d_P\mleft(\binom{\lambda_0\circ q_0}{\lambda_1\circ q_1}(x,i),\binom{\lambda_0\circ q_0}{\lambda_1\circ q_1}(y,j)\mright) \leq \gamma((x,i), (y,j)).
	\]
	The former equals $d_P(\lambda_i([x,i]), \lambda_j([y,j]))$.
	Recall that $\gamma$ is reflexive provided $q_0$ and $q_1$ are both sections of a morphism $d\colon S\to X$.
	In particular, $q_0$ and $q_1$ are regular monomorphisms in $\SMKH$.
	Thus, by \cref{l:pushout-mono},
	\[
		d_P(\lambda_i([x,i]), \lambda_j([y,j])) = \begin{cases}
			\gamma(x,y) & \text{if }i = j,\\
			\inf_{z \in X} \big(\gamma((x,i), (z,j)) + \gamma((z,i), (y, j))\big) & \text{if }i \neq j.
		\end{cases} \qedhere
	\]
\end{proof}

All told, we obtain a characterisation of equivalence continuous submetrics.

\begin{theorem}\label{t:equivalence}
	A binary continuous submetric $\gamma$ on a separated metric compact Hausdorff space $X$ is an equivalence continuous submetric if and only if, for all $x, y \in X$ and $i,j \in \{0,1\}$,
	\[
		d_X(x, y) \leq \gamma((x,i), (y,j)) = \gamma((x,i^*),(y,j^*))
	\]
	and
	\[
	\gamma((x,i),(y,i^*)) = \inf_{z \in X} \big(\gamma((x,i), (z,i^*)) + \gamma((z,i),(y,i^*))\big).
	\]
\end{theorem}

\begin{proof}
	By \cref{l:refl,l:symm,l:transitive}.
\end{proof}

\begin{remark}
	$\SMKHop$ is not a Mal'tsev category, since not every reflexive relation is an equivalence relation.
	Indeed, the following is an example of a reflexive non-symmetric binary continuous submetric on a one-point space $\{*\}$ (with $d(*,*) = 0$):
	\[
	\gamma((*,0), (*, 0)) = \gamma((*,0), (*, 1)) = \gamma((*,1), (*, 1)) = 0 \text{ and }
		\gamma((*,1),(*, 0)) = \infty.
	\]
	This corresponds to the surjection $\{(*,0), (*,1)\} \to \{a_1,a_2\}$, $(*,i) \mapsto a_i$, where
	$d(a_i,a_j)$ is $\infty$ if $i= 1$ and $j=0$, and $0$ otherwise.
\end{remark}

Dualising the definition of effective equivalence relations, an equivalence corelation $\binom{q_0}{q_1}\colon X+X \epi S$ on a separated metric compact Hausdorff space $X$ (and so the corresponding equivalence continuous submetric) is \emph{effective} if it coincides with the cokernel pair of its equaliser, i.e.\ if the following is a pushout in $\SMKH$
\begin{equation*}
	\begin{tikzcd}
		A \arrow[hookrightarrow]{r}{i} 	\arrow[hookrightarrow,swap]{d}{i}						& X \arrow{d}{q_1} \\
		X\arrow[swap]{r}{q_0}	& S,
	\end{tikzcd}
\end{equation*}
where $i \colon A \hookrightarrow X$ is the equaliser of $q_0,q_1\colon X\rightrightarrows S$ in $\SMKH$.

\begin{notation}
	For a closed subspace $A$ of a separated metric compact Hausdorff space $X$, we define the map $\gamma^{A} \colon (X+X) \times (X \times X) \to [0, \infty]$ by setting, for $x, y \in X$ and $i \in \{0,1\}$,
	\begin{align*}
	\gamma^A((x,i), (y,i)) & \coloneqq d(x, y),\\
	\gamma^A((x,i), (y,i^*)) & \coloneqq \inf_{a \in A} \big(d(x,a) + d(a,y)\big).
	\end{align*}
\end{notation}

\begin{lemma}\label{l:cokernel-inclusion}
	Let $X$ be a separated metric compact Hausdorff space.
	The binary continuous submetric on $X$ associated with the pushout in $\SMKH$ of an embedding $A \rmono X$ along itself is $\gamma^A$.
\end{lemma}

\begin{proof}
	This is an immediate consequence of \cref{l:pushout-mono}.
\end{proof}

\begin{lemma}\label{l:effective-char}
	Let $\gamma$ be an equivalence continuous submetric on a separated metric compact Hausdorff space $X$, and set
	\[
	A \coloneqq \{a \in X \mid \gamma((a,0), (a,1)) = 0\} = \{ a \in X \mid \gamma((a,1), (a,0)) = 0\}.
	\]
	 Then $\gamma$ is effective if and only if for all $x, y \in X$ and $i \in \{0,1\}$, we have
	\[
	\gamma((x,i), (y,i^*)) = \inf_{a \in A} \big(d_X(x,a) + d_X(a,y)\big).\]
\end{lemma}

\begin{proof}
	Let us endow $A$ with the induced topology and induced metric.
	Denoting by $\binom{q_0}{q_1}\colon X+X \epi S$ the binary corelation on $X$ associated with $\gamma$, we have
	\[
	A = \{a \in X \mid d(q_0(a), q_1(a)) = 0 = d(q_1(a), q_0(a))\} = \{a \in X \mid q_0(a) = q_1(a)\}.
	\]
	Therefore, the embedding $A \rmono X$ is the equaliser of $q_0,q_1\colon X\rightrightarrows S$ in $\SMKH$.
	Thus, $\gamma$ is effective if and only if the following is a pushout in $\SMKH$.
	\begin{equation*}
		\begin{tikzcd}
			A \arrow[hookrightarrow]{r}{} 	\arrow[hookrightarrow]{d}{}						& X \arrow{d}{q_1} \\
			X  \arrow[swap]{r}{q_0}	& S
		\end{tikzcd}
	\end{equation*}
	In turn, by \cref{l:cokernel-inclusion}, this is equivalent to saying that ${\gamma}={\gamma^A}$.
	By definition of $\gamma^A$, for all $x, y \in X$ and $i \in \{0,1\}$ we have
	\[
	\gamma^{A}((x,i),(y,i)) = d(x, y),
	\]
	and
	\[
	\gamma^A((x,i), (y,i^*)) = \inf_{a \in A} \big(d(x, a) + d(a, y)\big).
	\]
	By \cref{r:reflexive-is-emb}, $\gamma((x,i),(y,i)) = d(x, y)$.
	The desired fact follows.
\end{proof}

\begin{lemma}\label{l:idempotents-are-reflexive-up-to-iso}
	Let $X$ be a compact Hausdorff space.
	Let $\rho \colon X \times X \to [0, \infty]$ be a continuous function with respect to the upper topology of $[0, \infty]$, and suppose that for all $x, y \in X$ we have
	\[
		\rho(x, y) = \inf_{z \in X} \big(\rho(x, z) + \rho(z, y)\big).
	\]
	Then, setting $A \coloneqq \{x \in X \mid \rho(x, x) = 0\}$, we have
	\[
		\rho(x, y) = \inf_{a \in A} \big(\rho(x, a) + \rho(a, y)\big).
	\]
\end{lemma}

\begin{proof}
	By the triangle inequality, it is enough to prove the inequality $\geq$.

	Fix $x, y \in X$, and let us prove $\rho(x, y) \geq \inf_{a \in A} \big(\rho(x, a) + \rho(a, y)\big)$.

	If $\rho(x, y) = \infty$, the inequality is trivial.
	Let us then assume that $\rho(x, y) < \infty$.
	It is enough to show that the set
	\begin{equation} \label{eq:intersection}
	\{a \in A \mid \rho(x,y) = \rho(x,a) + \rho(a,y)\}
	\end{equation}
	is nonempty;
	for this, by compactness of $X$, it is enough to prove that this set is a codirected intersection of closed nonempty sets (see \cite[Theorem~3.1.1]{Engelking1989}).

    Let $\mathcal{V}$ be the set of closed subsets of $X$, and, for each $\lambda \in {]0, \infty]}$, set
    \[
    \mathcal{W}_\lambda \coloneqq \{K \in \mathcal{V} \mid \exists u,v \in K.\ \rho(u,v) \leq \lambda, \, \rho(x, u) + \rho(u, v) +\rho(v,y) = \rho(x,y)\}.
    \]

    Let $\mathcal{F}$ be the set of finite closed covers of $X$, i.e.\ the set of finite subsets $\mathcal{A}$ of $\mathcal{V}$ such that $\bigcup \mathcal{A} = X$.
    For every $\mathcal{A} \in \mathcal{F}$ and $\lambda \in {]0, \infty]}$, we set
    \[
    D^\mathcal{A}_\lambda \coloneqq \bigcup_{K \in \mathcal{A} \cap \mathcal{W}_\lambda} K.
    \]

    We now prove that the set $\{D^\mathcal{A}_\lambda \mid \mathcal{A} \in \mathcal{F},\, \lambda \in {]0, \infty]}\}$ is a codirected set of closed nonempty sets whose intersection is \eqref{eq:intersection}.

    For all $\mathcal{A} \in \mathcal{F}$ and $\lambda \in {]0, \infty]}$, the set $D^\mathcal{A}_\lambda$ is closed because it is a finite union of closed sets.

    The set $\{D^\mathcal{A}_\lambda \mid \mathcal{A} \in \mathcal{F}, \, \lambda \in {]0, \infty]}\}$ is codirected because $D^{\{X\}}_{\infty}$ belongs to it and, for all $\mathcal{A}, \mathcal{A}' \in \mathcal{F}$ and $\lambda, \lambda' \in {]0, \infty]}$, setting $\bar{A} \coloneqq {\{K \cap K' \mid K \in \mathcal{A},\, K' \in \mathcal{A}'\}}$ and $\bar{\lambda} \coloneqq {\min\{\lambda, \,\lambda'\}}$,
    \[
        D^{\bar{\mathcal{A}}}_{\bar{\lambda}} \subseteq D^\mathcal{A}_\lambda \cap D^{\mathcal{A}'}_{\lambda'}.
    \]

    We show that, for all $\mathcal{A} \in \mathcal{F}$ and $\lambda \in {]0, \infty]}$, the set $D^\mathcal{A}_\lambda$ is nonempty.
    We denote with $\sharp S$ the cardinality of a set $S$.
    Pick any natural number $l$ such that $\frac{l}{\sharp \mathcal{A}} > 1$ and $\frac{\rho(x,y)}{\frac{l}{\sharp \mathcal{A}} - 1} \leq \lambda$.
	Having fixed $x,y \in X$, we can use nonemptiness of $X$, and so for all $u, v \in X$ there is $z \in X$ such that $\rho(u,v) = \rho(u, z) + \rho(z, v)$, i.e.\ the infimum in the hypothesis of the lemma is a minimum.
    Thus, there are $z_1, \dots, z_{l} \in X$ such that
    \[
	    	\rho(x,y) = \rho(x,z_1) + \rho(z_1, z_2) + \dots + \rho(z_{l-1}, z_l) + \rho(z_{l}, y).
    \]
    Since every $z_i$ belongs to some $K \in \mathcal{A}$, we have $\sum_{K \in \mathcal{A}} \sharp (K \cap \{z_1, \dots, z_{l}\}) \geq l$.
    Therefore, the average of $\sharp(K \cap \{z_1, \dots, z_{l}\})$ for $K$ ranging in $\mathcal{A}$ is greater than or equal to $\frac{l}{\sharp \mathcal{A}}$. (The average makes sense because, by nonemptiness of $X$, $\sharp \mathcal{A} \neq 0$.)
    Therefore, there is $K \in \mathcal{A}$ with $\sharp (K \cap \{z_1, \dots, z_{l}\}) \geq \frac{l}{\sharp \mathcal{A}}$.
    Let $z_{i_1}, \dots, z_{i_n}$ (with $i_1 < \dots < i_n$) be an enumeration of the elements of $K \cap \{z_1, \dots, z_{l}\}$.
    Note that $n \geq \frac{l}{\sharp \mathcal{A}} > 1$ and so $n \geq 2$.
    We have
    \begin{align*}
	    	\rho(x,y) & = \rho(x,z_1) + \rho(z_1, z_2) + \dots + \rho(z_{l-1}, z_l) + \rho(z_{l}, y)\\
		& \geq \rho(z_{i_1}, z_{i_2}) + \dots + \rho(z_{i_{n-1}}, z_{i_{n}}).
    \end{align*}
    Therefore, the average of $\rho(z_{i_j}, z_{i_j})$ for $j$ ranging in $\{1, \dots, n-1\}$ is less than or equal to $\frac{\rho(x,y)}{n-1}$. (The average makes sense since $n \geq 2$ and so $n -1 \geq 1$.)
    Therefore, there is $j \in \{1, \dots, n-1\}$ such that
    \[
    \rho(z_{i_j}, z_{i_{j+1}}) \leq \frac{\rho(x,y)}{n - 1},
    \]
    and so
    \[
    \rho(z_{i_j}, z_{i_{j+1}}) \leq \frac{\rho(x,y)}{n - 1} \leq \frac{\rho(x,y)}{\frac{l}{\sharp \mathcal{A}} - 1} \leq \lambda.
    \]
    We have
    \begin{align*}
    		\rho(x,y) & \leq \rho(x,z_{i_j}) + \rho(z_{i_j},z_{i_{j+1}}) + \rho(z_{i_{j+1}}, y) && \text{by triangle inequality}\\
		& \leq \rho(x,z_1) + \rho(z_1, z_2) + \dots + \rho(z_{l-1}, z_l) + \rho(z_{l}, y)&&\text{by triangle inequality}\\
		& = \rho(x,y),
    \end{align*}
	and hence
	\[
		\rho(x,y) = \rho(x,z_{i_j}) + \rho(z_{i_j},z_{i_{j+1}}) + \rho(z_{i_{j+1}}, y).
	\]
	 Therefore, $K \in \mathcal{W}_\lambda$.
	 Thus, $\varnothing \neq K \subseteq D^\mathcal{A}_\lambda$, and hence $D^\mathcal{A}_\lambda \neq \varnothing$.

    We now prove
    \begin{equation} \label{e:int-is-desired}
        \bigcap_{\mathcal{A} \in \mathcal{F},\, \lambda \in {]0, \infty]}} D^\mathcal{A}_\lambda = \{a \in A \mid \rho(x,y) = \rho(x,a) + \rho(a,y)\}.
    \end{equation}

    The inclusion $\supseteq$ is immediate.

    Let us prove the converse inclusion, i.e.\ $\subseteq$.
    Let $z \in \bigcap_{\mathcal{A} \in \mathcal{F},\, \lambda \in {]0, \infty]}} D^\mathcal{A}_\lambda$.

    We first prove that $\rho(x,y) = \rho(x,z) + \rho(z,y)$.
    By way of contradiction, suppose this is not the case.
    Then, $\rho(x,y) < \rho(x,z) + \rho(z,y)$ (since the inequality $\geq$ holds by the triangle inequality).
    The function
    \begin{align*}
    		f \colon X \times X & \longrightarrow [0, \infty]\\
		(u,v) & \longmapsto \rho(x,u) + \rho(v,y)
    \end{align*}
    is continuous with respect to the upper topology of $[0,\infty]$ because $\rho$ is such.
    Since $f(z,z) = \rho(x,z) + \rho(z,y) > \rho(x,y)$, there is an open neighbourhood $U$ of $z$ such that, for all $u,v \in U$, $f(u, v) > \rho(x,y)$.
    Then, since $X$ is a compact Hausdorff space, there are closed subsets $K$ and $L$ of $X$ such that $z \in K \subseteq U$, $z \notin L$ and $K \cup L = X$.
    We have $K \notin \mathcal{W}_\infty$ because for all $u,v \in K$ we have $\rho(x,u) + \rho(u, v) + \rho(v, y) = f(u,v) + \rho(u,v) \geq f(u,v) > \rho(x,y)$.
    From $K \notin \mathcal{W}_\infty$ and $z \notin L$ we deduce $z \notin D^{\{K,L\}}_\infty$, a contradiction.
    Thus, $\rho(x,y) = \rho(x,z) + \rho(z,y)$.

    We now prove $z \in A$, i.e.\ $\rho(z,z) = 0$.
    By way of contradiction, suppose this is not the case.
    Choose $\lambda$ such that $0 < \lambda < \rho(z,z)$.
    Then, since $\rho$ is continuous, there is an open neighbourhood $U$ of $z$ such that $\rho(u,v) > \lambda$ for all $u,v \in U$.
    Then, since $X$ is compact and Hausdorff, there are closed subsets $K$ and $L$ of $X$ such that $z \in K \subseteq U$, $z \notin L$ and $K \cup L = X$.
    We have $K \notin \mathcal{W}_\lambda$ since $\rho(u,v) > \lambda$ for all $u, v \in X$.
    From $K \notin \mathcal{W}_\lambda$ and $z \notin L$ we deduce $z \notin D^{\{K, L\}} _\lambda$, a contradiction.

    By compactness, $\bigcap_{\mathcal{A} \in \mathcal{F},\, \lambda \in {]0, \infty]}} D^\mathcal{A}_\lambda$ is nonempty and thus, by \eqref{e:int-is-desired}, there is $a \in A$ such that $\rho(x,y) = \rho(x,a) + \rho(a,y)$.
\end{proof}

\begin{remark} \label{r:enough-reflexive}
	\Cref{l:idempotents-are-reflexive-up-to-iso} has the following corollary: given a closed idempotent endorelation $\prec$ on a compact Hausdorff space $X$, for every $x, y \in X$ with $x \prec y$ there is $a \in X$ such that $x \prec a \prec a \prec y$.
	Indeed, it is enough to apply \Cref{l:idempotents-are-reflexive-up-to-iso} to the continuous function $d \colon X \times X \to [0, \infty]$ defined by
	\[
	d(x,y) = \begin{cases}
		0 & \text{if $x \prec y$,}\\
		\infty & \text{otherwise.}
	\end{cases}
	\]
\end{remark}

\begin{theorem}\label{t:effective-equiv}
	Every equivalence corelation in $\SMKH$ is effective.
\end{theorem}

\begin{proof}
	Let $\gamma$ be an equivalence continuous submetric on $X \in \SMKH$.
	For $x,y \in X$, set $\rho(x,y) \coloneqq \gamma((x,0), (y,1))$ ($= \gamma((x,1), (y, 0))$, by \cref{l:symm}).
	Set
	\[
	A \coloneqq \{a \in X \mid \rho(a,a) = 0\}.
	\]
	In view of \cref{l:effective-char}, we shall show that, for all $x, y \in X$,
	\begin{equation*}
		\rho(x, y) = \inf_{a \in A}\big(d_X(x,a) + d_X(a,y)\big).
	\end{equation*}

	Here are some properties of $\rho$.
	\begin{enumerate}
		\item \label{i:d-less-than-rho}For all $x,y \in X$, $d(x,y) \leq \rho(x,y)$.
		\item \label{i:rho-triangle-1}
		For all $x,y,z \in X$, $\rho(x,y) \leq d(x,z) + \rho(z,y)$.
		\item \label{i:rho-triangle-2}
		For all $x,y,z \in X$, $\rho(x,y) \leq \rho(x,z) + d(z,y)$.
		\item \label{i:rho-cont}
		The function $\rho \colon X \times X \to [0, \infty]$ is continuous with respect to the upper topology of $[0,\infty]$.
		\item \label{i:idempotence}
		For all $x, y \in X$, $\rho(x, y) = \inf_{z \in X} \big(\rho(x, z) + \rho(z, y)\big)$.
	\end{enumerate}
	Indeed, \eqref{i:d-less-than-rho} follows from \cref{l:refl},
	\eqref{i:rho-triangle-1} and \eqref{i:rho-triangle-2} follow from the triangle inequality of $\gamma$ and the fact that $\gamma$ is below the coproduct metric, \eqref{i:rho-cont} follows from the continuity of $\gamma$, and \eqref{i:idempotence} follows from \cref{l:transitive} and the transitivity of $\gamma$.

	Therefore, we can apply \cref{l:idempotents-are-reflexive-up-to-iso} to $\rho$ and obtain
	\[
		\rho(x, y) = \inf_{a \in A} \big(\rho(x, a) + \rho(a, y)\big).
	\]

	Moreover, for all $x \in X$ and $a \in A$, we have
	\[
	\rho(x, a) \leq d(x,a) + \rho(a, a) = d(x, a) \leq \rho(x, a),
	\]
	and hence $\rho(x, a) = d(x, a)$; similarly, $\rho(a,x) = d(a,x)$.

	Therefore, for all $x, y \in X$,
	\[
		\rho(x,y) = \inf_{a \in A} \big(\rho(x, a) + \rho(a, y)\big) = \inf_{a \in A} \big(d(x, a) + d(a,y)\big). \qedhere
	\]
\end{proof}

Finally, from \cref{t:regular,t:effective-equiv} we obtain the main result:

\begin{theorem}\label{d:thm:3}
	$\SMKH$ is Barr-coexact.
\end{theorem}

Let us quickly point out that there is no hope for such a result to hold without separation, as it is already visible in the preordered case.

\begin{example}[Preorders are not Barr-exact]
  In the category of preordered sets and monotone functions, equivalence relations need not be effective.
  Indeed, the two maps from a singleton $\{*\}$ to a two-element set $\{a,b\}$ with $a \leq b$ and $b \leq a$ form an equivalence corelation on $\{*\}$ which is not effective.
\end{example}

\cref{d:thm:3} shows an algebraic trait of the dual of $\SMKH$.
As a negative result, we note that $\SMKH$ cannot be dually equivalent to a variety of \emph{finitary} algebras, since, by \cite[Corollary~4.30]{HN23}, every finitely copresentable object in $\SMKH$ is finite.
However, the following remains open to us:

\begin{question}
Is the category $\SMKH$ dually equivalent to a (possibly many-sorted) variety of (possibly infinitary) algebras?
\end{question}
Having shown that the complete category $\SMKH$ is coexact, this problem amounts (see \cite{Bor94a, AR94}) to the question of whether $\SMKH$ has a regular cogenerating set formed by regular injective objects.

\section{The symmetric and the ordered cases}
\label{sec:symm-order-cases}

Let us end this paper with a few remarks about symmetric metrics and compact ordered spaces.
Recall from \cref{d:rem:1} that $\SSMKH$ denotes the full subcategory of $\SMKH$ defined by the symmetric objects (i.e.\ those satisfying $d(x,y) = d(y,x)$), and that the inclusion functor $\SSMKH \rightarrow \SMKH$ has a right adjoint and a left adjoint.
This, together with the fact that $\SMKH$ is Barr-coexact (\cref{d:thm:3}), implies immediately the following result.

\begin{theorem}
   The category $\SSMKH$ is Barr-coexact.
\end{theorem}

\begin{remark}
  Building on the results of \Cref{s:quotient-objects}, it is easy to see that, in the symmetric case, the class of equivalence classes of surjections going out from $(X, d_X)$ is in bijection with the continuous \emph{symmetric} submetrics on $X$.
\end{remark}

\begin{remark}
  The category $\CompHaus$ of compact Hausdorff spaces, which can be identified with the full subcategory of $\PosComp$ defined by the symmetric objects, is coMal'tsev.
  One might then suspect that $\SSMKH$ is coMal'tsev too.
  However, this is not the case: $\SSMKHop$ is not a Mal'tsev category, as there are reflexive internal relations in $\SSMKHop$ that are neither symmetric nor transitive.
  An example is the following: let $X = \{a, b\}$ be a two-element discrete space, with metric $d(a,b) = d(b,a) = 1$ (and self-distances equal to $0$).
  Consider the following binary continuous submetric $\gamma$ on $X$ (i.e.\ a continuous metric below $d_{X + X}$): self-distances are $0$, and all other distances are $\infty$, except for the distances from $(a,0)$ to $(b, 1)$ and from $(b,1)$ to $(a,0)$ which are $1$.
  Comparing this to the ``ordered'' case, it seems that what breaks being Mal'tsev is the possibility of having more than just two possible values for the distances.
\end{remark}

Recall also from \cref{d:rem:2} that the inclusion $\PosComp \hookrightarrow \SMKH$ has both a right and a left adjoint.
This, together with the fact that $\SMKH$ is Barr-coexact (\cref{d:thm:3}), implies immediately the main result of \cite{AR20}:

\begin{theorem}
	The category $\PosComp$ is Barr-coexact.
\end{theorem}

While the proof in \cite{AR20} uses Zorn's lemma, here we have illustrated a choice-free proof (thanks to a choice-free proof of \cref{r:enough-reflexive}).

\section*{Acknowledgements}
We thank Guram Bezhanishvili for a discussion that made us realise an error in the proof of \cref{l:idempotents-are-reflexive-up-to-iso} in an earlier version of this paper.
We also thank the referee for their comments and suggestions which helped us to improve substantially the presentation of the paper.


\newcommand{\etalchar}[1]{$^{#1}$}

\end{document}